\documentclass[12pt]{amsart}

\usepackage[all,color]{xy}
\usepackage{pb-diagram}
\usepackage[mathscr]{eucal}
\usepackage{hyperref}
\hypersetup{
    colorlinks=true, %set true if you want colored links
    linktoc=all,     %set to all if you want both sections and subsections linked
    linkcolor=blue,  %choose some color if you want links to stand out
}
\usepackage{enumitem,titletoc}

\usepackage{xcolor}

\usepackage[normalem]{ulem}

\DeclareMathAlphabet{\mathpzc}{OT1}{pzc}{m}{it}

\usepackage{amsfonts}
\usepackage{amsmath}
\usepackage{amssymb}

\begin{document}

%\section{}
%\subsection{}
\newcommand{\ssub}{\subset\joinrel\subset}

\renewcommand{\thefootnote}{\fnsymbol{footnote}}
\newcommand{\starttext}{ \setcounter{footnote}{0}
\renewcommand{\thefootnote}{\arabic{footnote}}}
\renewcommand{\theequation}{\thesection.\arabic{equation}}
\newcommand{\be}{\begin{equation}}
\newcommand{\bea}{\begin{eqnarray}}
\newcommand{\eea}{\end{eqnarray}} \newcommand{\ee}{\end{equation}}
\newcommand{\N}{{\cal N}} \newcommand{\<}{\langle}
\renewcommand{\>}{\rangle}
\def\ba{\begin{eqnarray}}
\def\ea{\end{eqnarray}}
\newcommand{\PSbox}[3]{\mbox{\rule{0in}{#3}\includegraphics{#1}
\hspace{#2}}}

\def\v{\vskip .1in}

\def\al{\alpha}
\def\b{\beta}
\def\c{\chi}
\def\d{\delta}
\def\e{\epsilon}
\def\g{\gamma}
\def\l{\lambda}
\def\m{\mu}
\def\n{\nu}
\def\o{\omega}
\def\f{\varphi}
\def\r{\rho}
\def\si{\sigma}
\def\t{\theta}
\def\z{\zeta}
\def\k{\kappa}

\def\G{\Gamma}
\def\D{\Delta}
\def\O{\Omega}
\def\T{\Theta}

\newcommand{\dbar}{\overline{\partial}}
\newcommand{\ddbar}{\sqrt{-1}\partial\dbar}

\def\cA{{\mathcal A}}
\def\cB{{\mathcal B}}
\def\cC{{\mathcal C}}
\def\cD{{\mathcal D}}
\def\cE{{\mathcal E}}
\def\cF{{\mathcal F}}
\def\cG{{\mathcal G}}
\def\cH{{\mathcal{H}}}
\def\cI{{\mathcal I}}
\def\cK{{\mathcal K}}
\def\cL{{\mathcal L}}
\def\cM{{\mathcal M}}
\def\cN{{\mathcal N}}
\def\cO{{\mathcal O}}
\def\cP{{\mathcal P}}
\def\cp{{\mathcal p}}
\def\cR{{\mathcal R}}
\def\cS{{\mathcal S}}
\def\cT{{\mathcal T}}
\def\cV{{\mathcal V}}
\def\cX{{\mathcal X}}
\def\cY{{\mathcal Y}}
\def\A{{\mathbb{A}}}
\def\N{\mathbb N}
\def\Z{{\mathbb Z}}
\def\Q{{\mathbb Q}}
\def\R{{\mathbb R}}
\def\C{{\mathbb C}}
\def\P{{\mathbb P}}
\def\K{{\rm K\"ahler }}

\def\KE{{\rm K\"ahler-Einstein }}
\def\KEE{{\rm K\"ahler-Einstein}}
\def\Rm{{\rm Rm}}
\def\Ric{{\rm Ric}}
\def\Hom{{\rm Hom}}
\def\mod{{\ \rm mod\ }}
\def\Aut{{\rm Aut}}
\def\End{{\rm End}}
\def\osc{{\rm osc\,}}
\def\vol{{\rm vol}}
\def\Vol{{\rm Vol}}
\def\reg{{\rm reg}}
\def\sing{{\rm sing}}
\def\KE{{\rm K\"ahler-Einstein\ }}
\def\PSH{{\rm PSH}}
\def\Ad{{\rm Ad}}
\def\Lie{{\rm Lie}}
\def\ti\tilde
\def\ann{{\rm ann\,}}
\def\u{\underline}

\def\pl{\partial}
\def\na{\nabla}
\def\i{\infty}
\def\I{\int}
\def\p{\prod}
\def\s{\sum}
\def\dd{{\bf d}}
\def\ddb{\partial\bar\partial}
\def\sub{\subseteq}
\def\ra{\rightarrow}
\def\hra{\hookrightarrow}
\def\Lra{\Longrightarrow}
\def\lra{\longrightarrow}
\def\LA{\langle}
\def\RA{\rangle}
\def\L{\Lambda}
\def\diam{{\rm diam}}
\def\Diff{{\rm Diff}}
\def\Mod{{\rm Mod}}
\def\Hilb{{\rm Hilb}}
\def\depth{{\rm depth}}
\def\Ass{{\rm Ass}}
\def\us{{\underline s}}
\def\Re{{\rm Re}}
\def\Im{{\rm Im}}
\def\tr{{\rm tr}}
\def\det{{\rm det}}
\def\half{ {1\over 2}}
\def\third{{1 \over 3}}
\def\ti{\tilde}
\def\un{\underline}
\def\Tr{{\rm Tr}}
\def\Ker{{\rm Ker}}
\def\spec{{\rm Spec}}
\def\supp{{\rm supp}}
\def\Id{{\rm Id}}

\def\pz{\partial _z}
\def\pv{\partial _v}
\def\pw{\partial _w}
\def\w{{\bf w}}
\def\x{{\bf x}}
\def\y{{\bf y}}
\def\tet{\vartheta}
\def\dwplus{\D _+ ^\w}
\def\dxplus{\D _+ ^\x}
\def\dzplus{\D _+ ^\z}
\def\chiz{{\chi _{\bar z} ^+}}
\def\chiw{{\chi _{\bar w} ^+}}
\def\chiu{{\chi _{\bar u} ^+}}
\def\chiv{{\chi _{\bar v} ^+}}
\def\os{\omega ^*}
\def\ps{{p_*}}
\def\div{{\rm div}}

\def\hO{\hat\Omega}
\def\ho{\hat\omega}
\def\o{\omega}
\def\KSB{{\rm KSB}}
\def\CM{{\rm CM}}
\def\WP{{\rm WP}}

\def\[{{\bf [}}
\def\]{{\bf ]}}
\def\Rd{{\bf R}^d}
\def\Ci{{\bf C}^{\infty}}
\def\pl{\partial}
\def\sq{{{\sqrt{{\scalebox{0.75}[1.0]{\( - 1\)}}}}}\hskip .01in}
\newcommand{\dotcup}{\ensuremath{\mathaccent\cdot\cup}}

\newcommand{\ssubset}{\subset\joinrel\subset}

\newtheorem{claim}{Claim}[section]
\newtheorem{theorem}{Theorem}[section]
\newtheorem{proposition}{Proposition}[section]
\newtheorem{lemma}{Lemma}[section]
\newtheorem{conjecture}{Conjecture}[section]
\newtheorem{example}{Example}[section]
\newtheorem{definition}{Definition}[section]
\newtheorem{corollary}{Corollary}[section]
\newtheorem{remark}{Remark}[section]

\address{$^*$ Department of Mathematics, Rutgers University, Piscataway, NJ 08854}

\address{$^{**}$ Department of Mathematics and Computer Science, Rutgers University, Newark, NJ 07102}

\address{$^\dagger$Department of Mathematics and Computer Science,  Rutgers University, Newark, NJ 07102}

 \thanks{ Research supported in part by National Science Foundation grant  DMS-1711439, DMS-1609335 and  Simons Foundation Mathematics and Physical Sciences-Collaboration Grants, Award Number: 631318. }

\centerline{{\bf   CONTINUITY OF THE WEIL-PETERSSON POTENTIAL }\footnote{Research supported in part by National Science Foundation grant  DMS-1711439, DMS-1609335 and  Simons Foundation Mathematics and Physical Sciences-Collaboration Grants, Award Number: 631318. }}

\bigskip

\centerline{ \small  JIAN SONG$^*$, JACOB STURM$^{**}$, XIAOWEI WANG$^\dagger$}

 \bigskip
 \medskip

{\noindent \small A{\scriptsize BSTRACT}. \footnotesize $~$~~    
Let
$\mathcal{M}_{\rm KSB}$ ({\it resp.} $\overline{\cM}_\KSB$) be the
the moduli space   of $n$-dimensional K\"ahler-Einstein manifolds ({\it resp.} varieties) $X$ with $K_X$ ample.
We prove that the Weil-Petersson metric on $\mathcal{M}_{\rm KSB}$ extends uniquely to the projective variety $\overline{\cM}_\KSB$, as a closed positive current with continuous local potentials.   This generalizes  a theorem of Wolpert \cite{W1} which treats the case $n=1$, and also confirms a conjecture of Berman-Guenancia \cite{BG}. In addition, we derive uniform estimates for the volumes of sublevel sets of \KE potentials.}

\bigskip

\section{Introduction}

\indent The Deligne-Mumford moduli space $\mathcal{M}_g$ is a  quasi-projective variety  which parametrizes (isomorphism classes of) smooth curves of genus $g$. If $X\in \cM_g$, then the tangent space of $\cM_g$ at $X$ is the space of quadratic differentials $H^0(X,2K_X)$. The length of a tangent vector $\eta$  with respect to $\o_{\rm WP}$, the Weil-Petersson 
 metric,  is its $L^2$ norm with
respect to the hyperbolic metric on $X$, and  
measures the rate at which the complex structure changes in the direction $\eta$.
 The metric   has negative holomorphic sectional curvature and finite diameter.   Wolpert \cite{W1, W2} proved that $\omega_{\rm WP}$ extends to a K\"ahler current on $\overline{\mathcal{M}}_g$, 
 %the Deligne-Mumford compactification of ${\mathcal{M}}_g$, 
a projective variety  which parametrizes stable curves of  genus $g$, and he showed that locally on $\overline{\mathcal{M}}_g$ we have $\omega_{\rm WP}= \ddbar \varphi$ for some continuous plurisubharmonic function $\varphi$.

In dimension $n>1$, the generalization of $\cM_g$ is $\mathcal{M}_{\rm KSB}$, which parametrizes smooth projective varieties $X$ with $K_X$ ample (i.e. \KE manifolds with $K_X$ ample, by the theorem of Aubin, Yau). The Weil-Petersson metric $\o_{\rm WP}$ on the tangent space of $\mathcal{M}_{\rm KSB}$ 
%on the moduli space of 
%$n$-dimensional K\"ahler manifolds $X$ with $K_X$ ample (i.e. smooth varieties with $K_X$ ample) 
at a point $X\in \mathcal{M}_{\rm KSB}$ is the $L^2$-metric on  the space of harmonic $(0,1)$-forms with coefficients in the tangent bundle of the manifold $X$.  Here the theory is not as complete as the Riemann surface case, but much is known. An early key result was the computation of the holomorphic bisectional curvature of $\o_{\rm WP}$ which was achieved by Siu \cite{Si}.  A second important advance is the result of Schumacher \cite{Sch} which can be formulated as follows.
Suppose $\pi:\mathcal{X} \rightarrow B$ be  
a smooth family of \KE manifolds $X$ of dimension $n$ and 
$K_X$ ample.
%defining a morphism $B\ra {\cM}_\KSB$. 
The relative canonical bundle $K_{\mathcal{X}/B}$ has a canonical metric $h$, given by the \KE volume form on each fiber. 
The main result of \cite{Sch}  shows that $\o$, the curvature of $h$, is a positive $(1,1)$ form on $\cX$. Moreover, $\pi_*(\o^{n+1})=\o_{\rm WP}$. Let $\bar\pi:  \overline{\mathcal{X} }\rightarrow \overline{B }$ be a
flat family of \KE varieties  $X$ of dimension $n$ and $K_X$ ample (i.e. ``stable"" varieties $X$  - cf. \cite{BG} for background and precise defintions)
%defining a morphism $\overline{B}\ra 
%\overline{\cM}_\KSB$, 
extending the smooth family $\pi:\mathcal{X} \rightarrow B$. Then $\o$
 can be extended  to a closed positive current on 
 $\overline{\mathcal{X}}$ with analytic singularities. In \cite{SSW}, we show that this positive current has vanishing Lelong number and hence 
 $K_{\overline{\mathcal{X}}/\overline{B}}$,
 the relative canonical bundle, is nef on $\bar\cX$. As a consequence, the Weil-Petersson volume of 
$\overline{\cM}_\KSB$ 
%(which is the moduli space of \KE varieites) 
is a finite rational number.  
%The fibers of $\ti\pi$ (and the elements of $\bar\cM_{\rm KSB})$ may be characterized as algebraic varieties $X$ with $K_X$ ample which admit \KE metrics (cf \cite{BG}), 

The goal of this paper is generalize Wolpert's result \cite{W1, W2} to higher dimensions, i.e. to establish the continuity of the local K\"ahler potential of $\o_{\rm WP}$ on 
 the compactified moduli space $\overline{\cM}_\KSB$.  In the process, we establish uniform estimates for the volumes of sublevel sets for the \KE potential, which may be of independent interest.

In order to state our results precisely, we first review some necessary background.
By the classical results Aubin and Yau, \cite{Au, Y}, a \K manifold with $K_X$ ample has a unique negatively curved \KE metric.  In \cite{SSW}, we prove that the Gromov-Hausdorff completion of the moduli space of $n$-dimensional negatively curved K\"ahler-Einstein manifolds  is canonically identified with the KSB compatification of the moduli space of smooth canonically polarized manifolds. In addition, we show the local potentials of the Weil-Petersson metric are bounded: Let $\cL_{\CM}\rightarrow \overline{\cM}_\KSB$  be the CM line bundle (as described in  \cite{PX,SSW}) and let $h$ be a fixed smooth hermitian metric on $\cL_{\CM}$. Since $\cL_{\CM}$ is ample \cite{PX}, we can choose 
$h=h_{\rm FS}$
to be the Fubini-Study metric restricted to $\cL_{\CM}$ and
$\omega_{\rm FS}=\textnormal{Ric}(h_{\rm FS})$. Then $\omega_{\rm WP}$, the Weil-Petersson metric  on $\overline{\cM}_\KSB$, is the curvature of a hermitian metric 
\begin{equation}\label{1wp}
h_{\rm WP} = h_{\rm FS}e^{-\varphi_{WP}}
\end{equation}
 and
\begin{equation}\label{2wp}
\omega_{\rm WP}= \omega_{\rm FS}+ \ddbar \varphi_{WP} \in c_1(\cL_{\CM})
\end{equation}
for some $\varphi_{\rm WP} \in \textnormal{PSH}(\overline{\cM}_\KSB, \omega_{\rm FS}) \cap L^\infty(\overline{\cM}_\KSB)$, as shown in \cite{SSW}. Our main result is the following.

\begin{theorem}\label{main}  The hermitian metric $h_{WP}$ is continuous and 
$$\varphi_\WP \in C^0 \left( \overline{\cM}_\KSB \right).$$

\end{theorem}

In particular, the Weil-Petersson metric extends to a closed positive current $\omega_{\rm WP}$ on $\overline{\cM}_\KSB$, i.e., for any point $p\in \overline{\cM}_\KSB$, there exist an open neighborhood $U$ of $p$ and $\Phi_{WP} \in \textnormal{PSH}(U) \cap C^0(U)$ such that in $U$, 
$$\omega_{\rm WP} = \ddbar \Phi_{WP}. $$
\v

Theorem \ref{main}  confirms a conjecture of Berman and Guenancia \cite{BG}. The estimates we need use the approach of \cite{SSW}, but  sharper  bounds are required.  Our argument should generalize to the case of  stable families of klt pairs whose general fibres are not necessarily smooth. We conjecture that the Weil-Petersson current can still be  still defined for the moduli space of canonically polarized projective varieties with klt singularities and it can be extended to a K\"ahler current with continuous local potentials on the compactified KSB moduli space.  Let $(\cM_\KSB)^\circ$ be the smooth interior part of $\overline{\cM}_\KSB$. Then as we conjectured in \cite{SSW},  $((\cM_\KSB)^\circ, \omega_{\rm WP})$ should have finite diameter and its metric completion should be homeomorphic to $\overline{\cM}_\KSB$. The conjecture always holds for a stable family of K\"ahler-Einstein manifolds over a one-dimensional disk where the central fibre has only singularities of complete simple normal crossings \cite{T0, Ru1}. 

\v
Next we introduce the notions of stable varieties and stable families.
The  following  equivalent characterizations of ``stable variety" are established in \cite{BG}. A  projective variety $X$ is stable if $K_X$ is ample and  it satisfies any one of the following additional conditions.
\begin{enumerate}
\item $X^{\rm reg}$ has a \KE metric whose volume is $c_1(K_X)^n$. 
\item $X$ is smooth or $X$ has a worst semi-log canonical singularities.
\item $X$ is K-stable (in the sense of \cite{T1,D})
\
\end{enumerate}

Now we can define stable families as follows.

\begin{definition} \label{stabfib} A  flat projective morphism $\pi: \mathcal{X} \rightarrow B$ between normal varieties  a stable family  if the following holds.

\begin{enumerate}

\item The relative canonical  bundle $K_{\mathcal{X}/B}$ is $\mathbb{Q}$-Cartier and  is $\pi$-ample.

\medskip

\item The  fibres of $\pi$ are stable and the generic fiber is smooth.

\end{enumerate}

\end{definition}

\v

We let $B^\circ$ be  the set of smooth points of $B$ over which $\pi$ is smooth and $\mathcal{X}^\circ = \pi^{-1}(B^\circ).$
 For each $t\in B^\circ$, there exists a unique K\"ahler-Einstein metric $\omega_t\in c_1(\cX_t)$ on the fibre $\cX_t = \pi^{-1}(t)$ and one can define the hermitian metric on the relative canonical bundle $K_{\cX^\circ/B^\circ}$ by $h_t = (\omega_t^n)^{-1}.$ It is proved by Schumacher \cite{Sch} and Tsuji \cite{Ts}, using different methods, that
$$\textnormal{Ric}(h) = - \ddbar \log h$$
is nonnegative on $\cX^\circ$ and  it  is strictly positive for a nowhere infinitesimally trivial family.
It is further shown in \cite{Sch} that $h$ can be uniquely extended to a non-negatively curved singular hermitian metric on $K_{\cX/B}$ with analytic singularities. In \cite{SSW},  we prove that $h$ has vanishing Lelong number everywhere on $\cX$ and it tends $-\infty$ exactly at the non-klt locus of the special fibres  of $\pi: \cX \rightarrow B$. In particular, our result gives an analytic proof of Fujino's theorem 
\cite{Fu} that the relative canonical bundle $K_{\cX/B}$ is nef. 

 If we let $\cE_{slc}$ be the union of semi-log and log canonical locus of special fibres of $\cX$ (or equivalently, the non-klt locus), then by \cite{BG, S},
$$\cE_{slc} = \{ \phi = -\infty\}\ \   $$ 
where $\phi$ is the K\"ahler-Einstein potential.
Our next theorem  establishes the continuity of $h$  on $\cX\setminus \cE_{slc}$.

\begin{theorem} \label{main2} Let $\pi: \mathcal{X} \rightarrow B$ be a stable family of $n$-dimensional   canonical models over a normal variety $B$.  Let $\omega_t$ be the unique K\"ahler-Einstein metric on $\mathcal{X}_t$ for $t\in B^\circ$ and $h$ be the hermitian metric on the relative canonical sheaf $K_{\mathcal{X}^\circ/B^\circ}$ defined by
$$h_t =( \omega_t^n)^{-1}.$$
The curvature $\theta= \ddbar \log h$ of $(K_{\mathcal{X}^\circ/B^\circ}, h)$  extends uniquely to a closed nonnegative $(1,1)$-current on $\mathcal{X}$ satisfying the following.

\smallskip

\begin{enumerate}

\item $\theta$ has vanishing Lelong number everywhere in $\cX$. 

\item $\theta|_{\mathcal{X}_t}=\omega_t$ for $t\in B^\circ$.

\medskip

\item $h$ is continuous on $\cX\setminus \cE_{slc}$.

\end{enumerate}
Moreover, if all the fibres of $\pi: \mathcal{X} \rightarrow B$ have at worst log terminal singularities, then $h$    is continuous on $\mathcal{X}$.

\end{theorem}
The variety  $\overline{\cM}_\KSB$ is a coarse moduli space, and as such does not support a universal family, but (cf. \cite{SSW}) there exists a projective variety $B$, a finite cover $F: B\ra \overline{\cM}_\KSB$ and a stable family $\cX_B\ra B$ such that if $t\in B$ then
$[X_t]= F(t)$. Let $\cK(n,V)$ be the set of  all negatively curved \KE manifolds of dimension $n$ and volume at most $V$.  Matsusaka's big theorem implies that every $X\in \cK(n,V)$ lies in one of a finite number of such families. In other words, there is a single projective
family $\cX\sub \cB\times \P^{N_m}$, pluricanonically imbedded, which contains every element of $\cK(n,V)$. Let $\chi={1\over m}\o_{\rm FS}$ where $\o_{\rm FS} $ is the Fubini-Study metric on $\P^{N_m}$. For $X\in\cK(n,V)$ let

$$ \o_{_{X}}\ = \ \chi_{_{X}}+\sqrt{-1}\ddb\phi_{_{X}}
$$
denote the unique \KE metric on $X$
(\cite{BG, S, SSW})
. It is necessary to estimate decay rates of $\phi_X$  in order to understand the asymptotic behavior of the Weil-Petersson metric near the boundary of the KSB moduli space . 
More precisely, we shall need the following volume estimate for the sublevel sets of the K\"ahler-Einstein potentials.

% on semi-log canonical models. Let $\chi$ be a smooth K\"ahler metric on the relative canonical bundle $K_{\cX/\cS}$ and $\chi_t = \chi|_{\cX_t}$. Then the K\"ahler-Einstein metric $\omega_t$ can be expressed by 
%
%$$\omega_t= \chi_t+\ddbar \varphi_t$$
%
%for a unique $\varphi_t\in \textnormal{PSH}(\cX_t, \chi_t)$ for each $t\in \cS$ (\cite{S, SSW}). 

%\begin{theorem} \label{main3} Let $\pi: \mathcal{X} \rightarrow B$ be a stable family of $n$-dimensional  canonical models over a projective normal variety $B$. 
%Then
%there exists   $C>0$ such that for $0 \leq m \leq n $, $K> 1$ and all $t\in \cS $,

%\begin{theorem}\label{main3}. Fix a component of $\overline{\cM}_{\rm KSB}$ whose generic point represents a smooth manifold. Then there exists $C,K>0$ with the following property. Let $X\in $\overline{\cM}_{\rm KSB}$,
%and $\phi$ the uni

\begin{theorem}\label{main3} Fix $n,V>0$ and $\chi\ra\cB$ as above. Then there exists
$C=C(n,V,\chi)>~0$ such that for all negatively curved K\"ahler-Einstein manifolds $X$ of dimension $n$ and volume at most $V$ the following holds. 
For all $0\leq m\leq n$ and all $K\geq 1$ 
have

\begin{equation}\label{1levest}
 \I_{\{\varphi_{_{X}}<-K \}}\, ( \chi_{_{X}}+ \ddbar\varphi_{_{X}})^m \wedge \chi_{_{X}}^{n-m}  \leq  CK^ne^{-K/{(4n+2)}} .\ 
\end{equation}
Moreover,  the estimate  (\ref{1levest}) holds for all smoothable negatively curved K\"ahler-Einstein varieties of dimension $n$ and volume at most $V$.

\end{theorem}
Theorem \ref{main3} implies a uniform exponential decay for the volumes of sub-level sets of the K\"ahler-Einstein potentials. We believe that these estimates will be useful in understanding the incompleteness of the Weil-Petersson metric. One would also like to derive a geometric version of Theorem \ref{main3} by replacing the K\"ahler-Einstein potential by the distance function.

We give a brief outline of the paper. In section 2, we review semi-stable reduction and in section 3, we establish the $C^0$-estimate for the K\"ahler-Einstein potentials. In section 4, we prove Theorem \ref{main3} for the exponential decay of the volume measure of the level sets of the K\"ahler-Einstein potentials. In section 5, we give an alternative proof for an integral estimate derived in \cite{SSW}.  In section 6 and 7, we prove the continuity of Weil-Petersson potentials for stable families and complete the proof of Theorem \ref{main}. Finally, we prove Theorem \ref{main2} in section 8.  

\section{Stable families and semistable reduction}

In this section, we will review  the semi-stable reduction theorem and follow the discussion in \cite{SSW} (section 4) with some simplification. Let $\pi:\cX\ra B$ be a stable family with $\dim B = d \geq 1$. The theorem of 
Adiprasito-Liu-Temkin
\cite{ALT}
says that there is a smooth variety $S'$, a finite base change $B'\ra B$,
and a birational map $\Psi:\cX'\ra \cX\times_B B'$ such that the projection
$\pi':\cX'\ra \cX\times_B B'$ is a semistable reduction. By abuse
of notation, we shall write $B$ for $B'$ and $\cX$ for 
$\cX'\ra \cX\times_B B'$. We obtain a diagram

\v
\begin{equation}\label{diag}
\begin{diagram}
\node{\ti\cX} \arrow {e,t}{\Phi}\arrow{ese,r}{\ti\pi}\node{\mathcal{X}'} \arrow{se,l}{ \pi' }  \arrow{e,t}{\Psi}         \node{\mathcal{X} } \arrow{s,r}{\pi} \\
\node{}   \node{}   \node{B} 
\end{diagram}
\end{equation}
where $\ti\cX$ and $\Phi$ are to be defined later.
\v
Semistable reduction means that there are local coordinates 
$t=(t_1,...,t_d)$  on $B$  and $(x_1,...,x_{n+d})$ on $\cX'$ such that the  $t_i\circ \pi'$ are multiplicity free monomials
in $x$. We introduce  some additional notations for more precise statements. For simplicity, we assume that $B$ is a Eculidean ball in $\mathbb{C}^d$ and so we can choose $t=(t_1, ..., t_d)$ as global coordinates on $B$. 

Let
$$ H_i\ = \ \{t_i\circ\pi=0\}\ \sub\ \cX.
$$
To say  that $\pi: \cX'\ra B$ is a semistable reduction means that for $1\leq i\leq d$, 
%let $\cA_i$ be the set of irreducible components of $\Psi^*H_i$ and $\cD_i\sub\cA_i$. Then 
the Cartier divisor
$\Psi^*H_i$, the pullback of the Cartier divisor $H_i$, is a  divisor of simple normal crossings whose components $F_1,F_2,...$  all have multiplicity one. Note that these components may not be smooth since they can intersect themselves.

We write  
$$ \Psi^*H_i\ = \ \{t_i\circ\pi'=0\}\ = \  \s_{F\in \cA_i}F
$$
where $\cA_i$ is the set of irreducible components of $\Psi^*H_i$.
  Observe that 
$$\pi'(F) = \{t_i=0\}\ {\rm for\ all}\ F\in\cA_i, $$
so in particular, if $ i\not=j$,
$$  \cA_i\cap\cA_j\ = \phi .$$

There is a  disjoint partition $\cA_i = \cD_i\sqcup\cV_i$ such that the elements of $\cD_i$ are non-exceptional and those of $\cV_i$ are exceptional satisfying 
\begin{equation}\label{decoh0}
 \Psi^*H_i\ = \ H_i'\ + \ V_i'\ = \ 
\s_{D\in \cD_i}\, D\ + \ \s_{E\in \cV_i}E.
\end{equation}
Here $H_i'$ is the strict transform of $H_i$. 

Let ${\rm Exc}(\Psi)$ be the exceptional set of $\Psi$ and $\cF\sub {\rm Exc}(\Psi)$ the union of divisorial components of ${\rm Exc}(\Psi)$. Then $\cF$ can be written as a disjoint union as follows  
\begin{equation} \label{badset}
 \cF =  \cV\sqcup\cH,  % =\ \cV_1  \sqcup \cdots\sqcup\cV_d\sqcup\cH, 
\end{equation}
where $\cH$ is the set of horizontal exceptional divisors, that is,  $E\in\cH$
if and only if  $\pi'(E)=B$, and $\cV$ is the set of vertical exceptional divisors.

We can write $\cX'=\cup_\al U^\al$, a finite covering by coordinate neighborhoods and choose local coordinates
$x\in U$ with the following properties. Fix $\al$ and let $U=U^\al$. 
Let $I=\{1,2,...,n+d\}$. Then there exist a disjoint decomposition $I=A_1\sqcup A_2\sqcup ... \sqcup A_d\sqcup R$ such that 
$$A_i=D_i\sqcup V_i$$ 
and
$$ t_i\circ\pi'  = \p_{j\in A_i}x_j  = 
\p_{j\in D_i}x_j\cdot\p_{j\in V_i}x_j.
$$
Here $D_i$ and $V_i$ are the set of indices associated to $\cD_i$ and $\cV_i$ given by the following:   $j\mapsto
\overline
{\{x_j=0\}}$
defines maps $\n:A_i\ra\cA_i$, $D_i\ra \cD_i$ and $V_i\ra\cV_i$ which need not be one to one since if $j,k\in D_i$ with $j\not=k$, then
$
\n(j)=\n(k) = D$ if the divisor $D$ intersects itself.   

The following lemma gives the construction of for $\tilde \pi: \tilde\cX \rightarrow B $ in diagram (\ref{diag}).

\begin{lemma} \label{dscr}
There exist $\Phi: \ti\cX \ra \cX'$, a series of locally toric blow-ups, and local coordinates on open charts $\{\ti U^\al\}_\al$ of $\ti\cX$ such that
\begin{equation}\label{transh} 
(\Psi\circ\Phi)^*H_i  =  \ti H_i + \ti V_i =  
\s_{D\in \ti\cD_i} D + \s_{E\in \ti\cV_i}\,\textcolor{black}{{a_E}}E, 
\end{equation}
where $\ti H_i$, $\ti V_i$, $\ti\cD_i$ and $\ti \cV_i$ are defined for $\tilde\pi$ similarly as in (\ref{decoh0}),
and on each $\ti U=\ti U^\al$, 
$$ t_i\circ\ti\pi\ = \ \p_{j\in \ti D_i} x_j\cdot\p_{j\in \ti V_i}x_j^{a_{ij}},\ 
 \ \ D_i\sqcup V_i\ = \ \ti D_i\sqcup\ti V_i, \ \ {\rm and} \ \ \# \ti D_i\leq 1,
$$
where $a_{ij} \in \mathbb{Z}^{+}\cup\{0\}$,  $\tilde D_i$ and $\tilde V_i$ are the set of indices associated to $\tilde \cD_i$ and $\tilde \cV_i$.
Moreover, the map $\ti\pi:\ti\cX\ra B$ is flat.

\end{lemma}

\begin{proof}

 We use induction: Fix $i$ suppose there exists $\al$ such that in $U=U^\al$ we have  $\# D_i^\al>1$.  Choose $\al$ such that $\# D_i^\al$ is maximal.   Then we blow up the smooth variety $\cap_{j\in D_i^\al}\, \n(j)$ (smoothness is a consequence of the fact that $\# D_i^\al$ is maximal).  In other words,
if $D_i=D_i^\al=\{j_i,...,j_m\}$ then we make the changes of variables  of the form
$$ x_{j_1}\mapsto x_{j_1}\ \ {\rm and} \ \ x_{j_p}\mapsto x_{j_1}x_{j_p}\ \ {\rm if} \ \ p>1
$$

In these new coordinates, $j_1\notin \ti D_i$. Instead, $j_1\in \ti V_i$. Thus we have reduced $\# D_i$ by one. Continuing in this fashion we prove the first part of the lemma. To see that $\ti\pi:\ti\cX\ra B'$ is flat we observe that both $\ti \cX$ and $B'$ are smooth and $\ti\pi$ is equi-dimensional. The `miracle flatness theorem' (cf. 26.2.11. of \cite{V}) implies $\ti\pi$ is flat.   \end{proof}

We remark that in Lemma \ref{dscr}, $\n(j)\sub\ti \cX$ is in the strict transform of $H_i$ if and only if $j\in \ti D_i$. In other words, in the local coordinate chart $\ti U\sub\ti \cX$ at most one component of the total transform of $H_i$ is non-exceptional.  Lemma \ref{dscr} is an improvement of the procedure in \cite{SSW} and helps to simplify future calculations.

Finally, we recall the following adjunction lemma proved in \cite{SSW}.

\begin{lemma} There exist $a_k\leq 1$ and $b_l\leq 0$ such that
\begin{equation}\label{discrep2}
 K_{\ti \cX}+\s_{i=1}^d\ti H_i  =  
(\Psi\circ\Phi)^* \left(K_\cX+\s_{i=1}^d H_i \right) - \s_{E_k\in\ti\cV} a_kE_k 
- \s_{F_l\in \tilde \cH} b_l F_l .
\end{equation}

\end{lemma}

\section{The $C^0$-estimate for the K\"ahler-Einstein potential}

In this section, we will prove a sharp $C^0$-estimate for the K\"ahler-Einstein potential. We will follow the notations in section 2 and  replace $B$ (or $B'$) by $B$,  a Euclidean ball in $\mathbb{C}^d$.  Let $\eta_0,...,\eta_M$ be a basis for the pluricanonical system $|mK_{\cX/B}|$ for some sufficiently large $m\in \mathbb{Z}$ so that $mK_{\cX/B}$ is globally generated. Let
\begin{equation} \label{defforchi}
 \O\ = \ \left(\s_{j=0}^M |\eta_j|^2\right)^{1/m}\ \ {\rm and}\ \ \chi\ = \ \ddbar\log\O.
\end{equation}

We can immediately translate the algebraic adjunction formula (\ref{discrep2}) into the following analytic local formula with coordinates on $\ti\cX$. 
\begin{lemma} There exists $C>0$ such that on each open chart $\ti U= \ti U^\al$, 

\begin{equation}\label{den}
 \frac{ (\sqrt{-1})^{n+d}(\Psi\circ\Phi)^*\O_t \wedge_{i=1}^d 
 \left( \sqrt{-1} dt_i \wedge d\bar t_i \right) }{ |t|^2} \leq  
{C\wedge_{k=1}^{n+d} \left( \sqrt{-1} dx_k\wedge d\bar x_k \right)\over \p_{x_j\in \ti D}|x_j|^2\p_{x_j\in \ti V}|x_j|^{2a_j}},
\end{equation}
where $\{x_j=0\}$ corresponds to the divisor in $\ti \cD= \cup_i \ti \cD_i $ or $\ti\cV= \cup_i \cV_i$. 
\end{lemma}

For $t\in B$,  we let $\varphi_t$ be the solution to the following complex Monge-Ampere equation
\begin{equation}\label{keqn}
 (\chi_t+\ddbar\varphi_t)^n\ = \ e^{\varphi_t}\O_t,
\end{equation}
where $\chi_t$ and $\O_t$ are the restrictions of $\chi$ and $\O$ to $\cX_t=
\pi^{-1}(t)$.  By \cite{S}, equation (\ref{keqn}) admits a unique solution $\varphi_t$ with vanishing Lelong number along the semi-log and log locus of $\cX_t$ for each $t\in B$. 
For $t\in B$, $\o_t=\chi_t+ \ddbar \varphi_t$ is the K\"ahler-Einstein metric on $\cX_t$  satisfying
$$\textnormal{Ric} (\o_t) = - \o_t,$$
including those $t$ for which  $\cX_t$ is a singular semi-log canonical model. 
We define $\cX^{\textnormal{reg}}$ be the union of smooth points of $\cX_t$ for all $t\in B$ and $\varphi \in L^{\infty}_{loc}(\cX^{\textnormal{reg}})$ such that
$$\varphi|_{\cX_t} = \varphi_t. $$
It is proved in \cite{SSW} that $\varphi$ extends globally to $\cX$ in  $\textnormal{PSH}(\cX, \chi)$ with vanishing Lelong number. In particular, $\varphi$ is bounded above,  and locally bounded below away from the non-klt locus of $\cX_t$ for all $t\in B$. In fact, $\chi+\ddbar \varphi$ is the curvature of the relative canonical line bundle $K_{\cX/B}$.  

Let $\cE$ be any divisor of $\ti\cX$   containing the exceptional locus of $\Psi\circ\Phi$. Theorem 4.7 of \cite{ALT} says that we may choose  $\Psi$ in (\ref{diag}) to be an isomorphism over any open subeset of $B$ for which $\pi$ is smooth (in particular in (\ref{badset}) we have $\cH=\emptyset$). Thus, we may choose $\cE\sub\ti\cX$ in such a way
that  

\be\label{Echoice}
B\backslash B^\circ\sub \ti\pi(\cE)\ \ {\rm and}\ \  \ti\pi(\cE)\sub B\ \ 
{\rm has\ codimension\ one}
%\ti\pi(\cE)\sub Z\sub  B\ \ \hbox{where $Z\sub B$ is a divisor.}
\ee

\v

We let $\sigma_\cE$ be a defining section of $\cE$ and $h_\cE$ be a smooth hermitian metric on the line bundle on $\tilde{\cX}$ associated to $\cE$. Then the function $\log |\sigma_\cE|^2_{h_\cE}$ is defined on $\tilde\cX$ and by abusing the notations, we identify it with its push-forward onto $\cX$. The following is the main result of this section and it is a sharp improvement of the $C^0$-estimate in \cite{S, SSW}.

\begin{proposition}\label{C0} For any $\epsilon>0$, there exists $C_\epsilon>0$ such that on $\cX$, we have
\begin{equation}\label{mainest}
 \varphi \ \geq \ -(2n+\e)\log \left(-\log|\si_{\cE}|^2_{h_\cE} \right)\ - \ C_\e.
\end{equation}

\end{proposition}
\noindent
Remark: If we choose $\cE_1,...,\cE_p$ as in (\ref{Echoice}) such that
$  B\backslash B^\circ\ = \ \bigcap_{i=1}^p\, \ti\pi(\cE_i)
$
we obtain
\begin{equation}\label{mainest2}
 \varphi \ \geq \ -(2n+\e)\log \left(-\log\left(\s_{i=1}^p|\si_{\cE_i}|^2_{h_\cE}\right) \right)\ - \ C_\e.
\end{equation}
Note that the right side of (\ref{mainest2}) is finite on each smooth fiber.
\v
In the case where the base has dimension $1$ and the central fibre is irreducible, a  sharper estimate is achieved in \cite{DGG} with  a coefficient of $-(n+1+\e)$. However, any constant will suffice for our application to the estimates of  Weil-Petersson potentials.

Before proving Proposition \ref{mainest}, let us have a quick review on a standard procedure for constructing plurisubharmonic functions from convex functions.

\begin{definition}
A continuous convex function $H: \R^-  \rightarrow \R$ is said to be an admissible function if
$$H'>0, ~ H''>0$$
on $\R^-$. 

\end{definition}

For example, $H(x)= - \log (-x)$ is an admissible function and $H(x) = - (-x)^{1-\delta}$ is an admissible function whenever $\delta\in (0,1)$. 
 
\begin{lemma} Suppose $\phi$ is a negative plurisubharmonic function on a  domain in $\C^n$ and $H$ is an admissible function. Then $H\circ \phi$ is also plurisubharmonic.

\end{lemma}

\begin{proof} Straightforward calculations show that
$$\ddbar (H \circ \phi)= H' \ddbar\phi + \sqrt{-1} H'' \partial \phi \wedge \dbar \phi \geq 0.$$

\end{proof}

  If $v=\log |z|^2$ on the unit ball in $\mathbb{C}$ and $F_\delta(x)=- (-x)^{1-\d}$ for some $0<\delta<1$, then
$$F_\delta'(x)=(1-\d)({-}x)^{-\d}, ~~F_\delta''(x)=\d(1-\d)(-x)^{-1-\d}.$$
 On  $\C^*$, we have
  $$ \ddbar\big(-(-\log |z|^2)^{1-\d}\big)\ \ = \ \d(1-\d)\cdot
{\sqrt{-1} dz\wedge d\bar z\over (-\log|z|^2)^{1+\d}\cdot |z|^2}
$$
Let  $\varphi$ be a smooth function and $v=\log (|z|^2e^{-\varphi})= \log|z|^2 - \varphi$ on the unit ball,    then  for $|z|$ sufficiently small, we have
\begin{equation}\label{one} 
\ddbar\big(-(-v)^{1-\d}\big)\ \ \geq \ {1\over 2} \d(1-\d) 
{\sqrt{-1}dz\wedge d\bar z\over (-\log\left( |z|^2 e^{-\varphi}\right) )^{1+\d}\cdot \left(|z|^2e^{-\varphi}\right) } .
\end{equation}

If $\chi+\ddbar u>0$ for some K\"ahler form $\chi$ and if $H'(u) \leq 1$ for some admissible function $H$, then
\begin{equation}
\label{two} \chi+\ddbar H(u)\ \geq  \chi+ 
H'(u)\ddbar u  \geq  H'(u)(\chi+\ddbar u).
\end{equation}

\medskip

We now begin to prove Proposition \ref{C0}.

\medskip

\noindent {\it Proof of Proposition \ref{C0}.}   We will follow the same approach in \cite{SSW} with some simplification.  Let $ {\rm Exc}(\Psi\circ\Phi)$  be the exceptional locus in the diagram (\ref{diag}) and Lemma \ref{dscr} for $\Psi\circ \Phi$.
By Kodaira's  lemma, there exists an effective $\mathbb{Q}$-Cartier divisor $E$ whose support  contains $ {\rm Exc}(\Psi\circ\Phi)$ and 
the class $(\Psi\circ \Phi)^*[\chi] - [E]$ is ample on $\ti\cX$. Furthermore, we let $\sigma_E$ be a defining section of $E$ and $h_E$ be a smooth hermitian metric on $\ti\cX$ such that
$$|\sigma_E|^2_{h_E} \leq 1/2$$
and
$$(\Psi\circ \Phi)^*\chi + \ddbar \log |\sigma_E|^2_{h_E} >0$$
is a K\"ahler form on $\ti\cX$. 

We now define for $\e>0$ sufficiently small,
\begin{equation}
u_\epsilon  = 
\e  \log|\si_E|^2_{h_E} \ - \e^2  
\left(-\log|\si_E|^{2}_{h_E}\right)^{1-\d}  -3n \leq  -3n
\end{equation}
and
\be\label{fepsilon} f_\e = H (u_\epsilon) = H \left(\e  \log|\si_E|^2_{h_E} \ - \e^2  
\left(-\log|\si_E|^{2}_{h_E}\right)^{1-\d}  -3n  \right).
\ee
It is straightforward to verify that there exist $\e_0$ and $\delta_0>0$ such that for any $0<\e<\e_0$ and $0<\delta<\delta_0$, we have 

$$u_\e \in \textnormal{PSH}(\cX, \chi)$$
since $\left(-\log|\si_E|^{2}_{h_E}\right)$ is bounded below by  $\log 2> 0$.

The admissible function $H$ will also be chosen later and it has to satisfy the condition that $H'(u_\e) \leq 1$.

We let 
$$\varphi_\e=\varphi-f_\e, ~~\varphi_{t, \epsilon}= \varphi_\epsilon|_{\cX_t}=\varphi_t - f_{t,\e}, ~~f_{t, \e} = f_\e|_{\cX_t} .$$ 
Then the complex Monge-Ampere equation (\ref{keqn}) can be rewritten as
\begin{equation}\label{modkeqn}
 (\chi_t+\ddb f_{t,\e}+ \ddbar\varphi_{t,\e})^n\ = \ e^{\varphi_{t,\e}}e^{f_{t,\e}}\O_t
\end{equation}

Note that since $\varphi_\e$ tends to $\infty$ near the exceptional locus of $\Psi\circ \Phi$ or $E$, we may conclude that $\varphi_\e$ achieves its minimum. Suppose $\varphi_{t, \e}$ achieves its minimum at $p$, we have at $p$, 
\begin{equation} 
e^{\varphi_{t,\e}}   \geq 
\frac{(\sqrt{-1})^d |t|^{-2}e^{-f_{t,\e}}
(\chi_t+\ddbar f_{t,\e})^n
   \wedge_{i=1}^d dt_i\wedge d\bar{t}_i} 
{(\sqrt{-1})^d|t|^{-2} \O_t    \wedge_{i=1}^d dt_i\wedge d\bar{t}_i} \\
%
%& \geq & \left({e^{-f_{t,\e}} (\chi+\ddbar f_{t,\e})^n \wedge{dt\wedge d\bar t\over |t|^2} } \right)
%\left({dx\wedge d\bar x\over \p_{x_j\in \ti D}|x_j|^2\p_{x_j\in \ti V}|x_j|^{2}}\right)^{-1}
%
\end{equation}
Applying (\ref{one}) and (\ref{two}), there exists $c_1, c_2>0$ such that 
\begin{eqnarray*}
&&{1\over H'(u_\e)}(\chi+\ddbar f_\e) \\
& \geq&  \chi+\ddbar u_\epsilon \\
&\geq& c_1\sqrt{-1} \s_{i=1}^d\left(\s_{j\in \ti D_i}dx_j\wedge d\bar x_j +  \s_{j\in\ti V_i} {dx_j\wedge d\bar x_j
\over (-\log|x_j|^2)^{1+\d} |x_j|^{2}}\ 
\right)+c_1\sqrt{-1} \s_{j=1}^{n+d} dx_j\wedge d\bar x_j \\
& \geq& c_2 \sqrt{-1}  \s_{i=1}^d\left[ \s_{j\in\ti V'_i} {dx_j\wedge d\bar x_j\over 
(-\log|x_j|^2)^{1+\d}|x_j|^{2}}\ 
\right]+ c_2\sqrt{-1} \s_{j=1}^{n+d} dx_j\wedge d\bar x_j 
\end{eqnarray*}
where $\ti V_i'\sub \ti V_i$ with equality if and only if $\ti D_i\not=\phi$, otherwise, $\ti V_i'$ equals $\ti V_i$ with one element removed.
On the other hand,
$$ {dt_i\over t_i}  =  \s_{ j\in\ti D_i} {dx_j\over x_j}\ +
\s_{j\in\ti V_i} {dx_j\over x_j}
$$
and
$$ (\chi+\ddbar f_\e)^n  \geq   (H'(u_\e))^n(\chi + \ddbar u_\epsilon)^n.
$$
Therefore,
\begin{eqnarray*}
&& \frac{(\sqrt{-1})^d e^{-f_\e}(\chi+\ddbar f_\e)^n
  \wedge_{i=1}^d dt_i\wedge d\bar t_j }{|t|^2}\\ 
&\geq &
 e^{-H(u_\e)}{(H'(u_\e))^n\over \p_{j\in \ti V'}(-\log|x_j|^2)^{1+\d}}\cdot{ (\sqrt{-1})^{n+d} \wedge_{i=1}^{n+d} dx_i \wedge d\bar x_i \over
\p_{j\in \ti D}|x_j|^2\p_{j\in \ti V}|x_j|^{2},
} 
\end{eqnarray*}
where $\ti V'=\cup_{i=1}^d \ti V_i'$,  $\ti D'=\cup_{i=1}^d \ti D_i'$.
Immediately, there exists $C>0$ such that   
\begin{eqnarray*}
 e^{\varphi_\e} &\geq& {e^{-H(u_\e)}H'(u_\e)^n\over \p_{j\in \ti V'}(-\log|x_j|^2)^{1+\d}} \\
&\geq& { e^{-H(u_\e)} (H'(u_\e))^n\over C |u_\epsilon|^{\textcolor{black}{n(1+\d)}}}.
\end{eqnarray*}
If we now let $$H(x)=- (2n+\e)\log (-x),$$ then 
$$H'(u_\epsilon) ={2n+\e\over |u_\e |}<1$$ and at the minimal point $p$, we have 

$$e^{-H(u_\e)}{H'(u_\e)^n\over |u_\e|^{n(1+\d)}}\
 \geq \
 |u_\e|^{2n+\e}{1\over |u_\e|^{2n+n\delta}}  = \frac{1}{|u_\e|^{n\delta-\e}}\geq   1,
$$
\textcolor{black}{if we choose $\e > n\delta $. }. 
Therefore there exists $C_\e>0$ such that
$$\varphi_\e \geq -C_\e$$
or equivalently,
\begin{equation}\label{4522}
\varphi\geq   f_\epsilon - C_\e.
\end{equation}

This completes the proof of Proposition \ref{mainest} by combining (\ref{4522}) and the definition of $f_\epsilon$ in (\ref{fepsilon}). $ \hfill  \Box$

\section{Integrability of $\log|\si_\cE|_{h_\cE}^2$}

Let $\si_\cE$ and $h_\cE$ be the holomorphic section and hermitian metric of $\cE$ in Proposition \ref{C0}. For $t\in B'=B\backslash Z$ We will prove a uniform estimate for the integrals of $\log |\sigma_\cE|^2_{h_\cE}$. This can be proved by local calculations from the argument in \cite{SSW}. In this section, we give an alternative proof using 
a result of Morikawi (\cite{M}).

\begin{proposition}\label{mor}  For $t\in B'=B\backslash Z$, let
\begin{equation}\label{integr}
I(t)\ =  \I_{\cX_t} (-\log|\si_\cE|_{h_\cE}^2)\,\chi_t^n. 
\end{equation}
Then $I(t)$
extends to a continuous function $B\ra\R$. In particular, $I(t)$ is uniformly bounded for all $t\in B$. 
\end{proposition}
The proof is an application of the following  theorem of  Morikawi (\cite{M}).

\begin{theorem}\label{Mor} Let $f_X: X\ra B$ be a flat projective morphism of algebraic varieties over  $\C$ of relative dimension $n$, let $L_0,...,L_n$ be line bundles equipped with smooth hermitian metrics $h_0,...,h_n$. Then
$\langle h_0,...,h_n\rangle $ is a continuous hermitian metric on the Deligne pairing $\langle L_0,...,L_n\rangle(X/B)\ra B$. Thus, if $l_0,...,l_n$ are generic rational sections of $L_0,...,L_n$  such that  $\cap_{j=0}^n\, \div(l_j)=\phi$, then $|\langle l_0,...,l_n\rangle|$ is a positive continuous function on $B$.
\end{theorem}
\noindent
Note that the fibers $X_t$ are all projective varieties, but $X$ and $B$ need not  be projective. 
\v
For details on the definition and properties of Deligne pairings see 
\cite{Z} and \cite{M}, but let us here recall the basic framework. If $f_X: X\ra B$ is as above, and  $L_0,...,L_n$ are hermitian line bundles on $X$, then
$\langle L_0,...,L_n\rangle(X/B)$ is a hermitian line bundle on $B$. For example, if $n=0$ then $\<L_0\>$ is the ``norm" of $L=L_0$. Thus, the fiber
of $\<L\>$ at $b\in B$ equals $\otimes_{x\in f_X^{-1}b}L_x$   and if $l$ is a rational section of $L$ then 
$\<l\>=\otimes_{x\in f_X^{-1}b}l_x$
 is a rational section of $\<L\>$  and  
$|\<l\>|_{\<h\>}=\otimes_{x\in f_X^{-1}b}|l_x|_h$. In
general, $\langle L_0,...,L_n\rangle(X/B) = \langle L_0,...,L_{n-1}\rangle(\div(l_n)/B)$, and the restriction of 
$\langle l_0,...,l_n\rangle$ to $\div(l_n)$ is just
$\langle l_0,...,l_{n-1}\rangle$. Moreover the formula for the norm of the restriction is given below by (\ref{del}). In fact these properties give an inductive characterization of the Deligne pairing.

\v
Let $f_X:X\ra B$ be as in Theorem \ref{Mor}, let $B^\circ =\{t\in B\,:\, f_X^{-1}(t)$ is smooth$\}$ and $B'\sub B^\circ$ a Zariski open subset.
We shall    need the following lemma.

\begin{lemma}\label{mor1}  Fix $l_0,...,l_n$ rational sections of $L_0,...,L_n$ such that  $\cap_{j=0}^n\, \div(l_j)=\phi$.  Assume that one of the following conditions holds
\begin{enumerate}

\item $\div(l_0)\ra B$ is flat or 

\medskip

\item $f(\div(l_0))\cap B'=\phi$ and  $\div(l_j)\ra B$ is flat for
$1\leq j\leq n$. 
\end{enumerate}
 Then the map $I': B'\ra\R$ given by 
$$t\mapsto \I_{X_t} \log|l_0|^2_{h_0}\,c_1(h_1)\wedge\cdots c_n(h_n)$$
is continuous and extends to a continuous function $B\ra\R$.

\end{lemma} 

\begin{proof} Assume $\div(l_0)\ra B$ is flat. Then the result follows from Theorem \ref{Mor} and the induction formula for Deligne pairings:

\be\label{del}
 \log|\langle l_0,...,l_n\rangle(X/B)|\ = \ \log|\langle l_1,...,l_n\rangle(\div(l_0)/B)|\ - \ \I_{X_t} \log|l_0|^2_{h_0}\,c_1(h_1)\wedge\cdots c_n(h_n)
\ee
on $B'$.
\v
Now assume $f(\div(l_0))\cap B'=\phi$ and
$\div(l_j)\ra B$ is flat for
$1\leq j\leq n$. 
 We proceed using induction on $n$.  Let $Y=\div(l_1)$. Then $f_Y:Y\ra B$ is flat with fiber dimension $n-1$. Writing $[L_j]:= \div(l_j)$ we compute, for $t\in B'$,

$$\I_{X_t} \log|l_0|^2_{h_0}\,c_1(h_1)\wedge\cdots c_n(h_n)\ = \ 
\I_{X_t} \log|l_0|^2_{h_0}\,([L_1]-\ddbar\log|l_1|^2_{h_1}))\wedge c_1(h_2)\wedge\cdots c_n(h_n)
$$

$$=\ \I_{Y_t}\log|l_0|^2_{h_0}\,c_1(h_2)\wedge\cdots c_n(h_n)\ - \ 
\I_{X_t}\log|l_1|^2([L_0]-c_1(h_0))\wedge c_1(h_2)\wedge\cdots c_n(h_n)
$$
Now the first integral is continuous by induction and the second is
continuous by part (1) of the lemma (note that $[L_0]\cap X_t=\phi$ for $t\in B'$ by assumption).

\end{proof}

\noindent {\it Proof of Proposition \ref{mor}.}  We wish to apply Lemma \ref{mor1} as follows (see (\ref{diag}) for the notation). Let $X = \ti \cX$, $L_0=O(\cE)$, $l_0=\si_\cE$, and $L_1=L_2=\cdots =L_n=(\Psi\circ\Phi)^* (mK_{\cX/B}$).  Then the hypotheses of part (2) of Lemma \ref{mor1} apply by virtue of (\ref{Echoice}).    ~~~~$\Box$

\section{Level sets of the K\"ahler-Einstein potentials}

Let $\cX\ra B$ be a stable family of canonically polarized manifolds,  where $B\sub\C^d$ is a Euclidean ball. Let
$$\hbox{$B^\circ=\{t\in B\,:\, \cX_t$ is smooth $\}$}
$$
and 
$$\cX^\circ =\pi^{-1} (B^\circ). $$
For $t\in B^\circ$, we let $\varphi_t$ be the solution of the complex Monge-Amp\`ere equation induced by the K\"ahler-Einstein equation satisfying 
$$\o_t^n=(\chi_t+\ddbar\varphi_t)^n=e^{\varphi_t}\O_t,$$
where $\chi_t$ and $\Omega_t$ are defined in (\ref{defforchi}).  The main goal of this section is to bound the volume of the level sets of the K\"ahler-Einstein potentials $\varphi_t$.

We first recall the definition of capacity in pluripotential theory.

\begin{definition} Let  $X$ be an $n$-dimensional compact K\"ahler manifold and let $\omega$ be a smooth closed nonnegative $(1,1)$-form on $X$. The capacity for a Borel subset $K$ of $X$ associated to $\omega$ is defined by

\begin{equation}
\textnormal{Cap}_\omega(K) = \sup\left\{ \int_K (\omega+ \ddbar u)^n
 ~|~ u \in \textnormal{PSH}(X, \omega),  ~ -1\leq u \leq  0\right\}. 
\end{equation}

\end{definition}

The following lemma is well-known \cite{Ko}.  
 \begin{lemma} \label{cap2} Suppose $X$ is a compact complex manifold, $\omega \geq 0$ a closed nonnegative $(1,1)$-form and $\psi\in \textnormal{PSH}(X,\o)\cap L^\infty(X)$ with $\psi\leq 0$. Then for $K\geq 1$, we have 
$$  \I_{\{\psi\leq-K\}}\,(\o+\ddbar\psi)^n\leq K^n \ {\rm Cap}_\o(\psi\leq-K)
$$
 \end{lemma}
  
\begin{proof} Let $\psi_K = \max(\psi, -K)$  and $u_K= K^{-1} \psi_K$. Then $u_K \in \textnormal{PSH}(X, \o)$ and $$-1\leq u_K \leq 0.$$
Then 
\begin{eqnarray*}
\int_{\psi\leq-K}(\o+\ddbar\psi)^n &=& \int_X \o^n - \int_{\psi> -K} (\o+\ddbar\psi)^n\\
&=& \int_X \o^n - \int_{\psi> -K} (\o+\ddbar\psi_K)^n\\
&=& \int_{\psi \leq -K} (\o+\ddbar\psi_K)^n\\
&\leq& K^n  \int_{\psi \leq -K} (\o+\ddbar u_K)^n\\
&\leq& K^n \ {\rm Cap}_\o (\psi\leq-K).
\end{eqnarray*}

\end{proof}

The following lemma is proved in \cite{GZ1} (Proposition 2.6).

\begin{lemma} \label{cap1} Let $\varphi\in \textnormal{PSH}(X,\o)$ with $\varphi\leq 0$. Then
for $K\geq 1$
$$ {\rm Cap}_\o(\varphi<-K)\ \leq K^{-1} \left( \I_X(-\varphi)\o^n\ + n \int_X \o^n \right)
. 
$$

\end{lemma}

The following proposition is the main result of  the section and is equivalent to Theorem \ref{main3}. It implies the exponential decay for the measure of the level set of the potential $\varphi_t$.

\begin{proposition}\label{expdec}
There exists   $C>0$ such that for $0 \leq m \leq n $ and all $t\in B^\circ$,
$$ \I_{\{\varphi_t<-K \}}\, ( \chi_t+ \ddbar\varphi_t)^m \wedge \chi_t ^{n-m}  \leq  CK^ne^{-K/{(4n+2)}}
$$
\end{proposition}

\begin{proof} The uniform $C^0$-estimate of Proposition \ref{C0} implies that there exists $C_1>0$ such that for all $t\in B^\circ$, 
$$ \varphi_t  \geq  -(2n+1)\log(-\log|\si_{\cE}|^2_{h_\cE})|_{\cX_t} - C_1 = -(2n+1)\log(-\t_t) - C_1
.
$$
We define $\psi_t$ by
$$\psi_t =-2(2n+1)\log(-\t_t)\ - 2C_1.$$
By the upper bound estimate for $\varphi_t$ in \cite{SSW}, there exists $A>0$ such that for all $t\in B$, 
$$\sup_{\cX_t} \varphi_t \leq A. $$
Then $\sup_{\cX_t} \psi_t \leq A.$
We let 
$$\psi_{t, M_t}=\max (\psi_t , -2M_t),$$ where $-M_t << \inf_{\cX_t} \varphi_t$.
Obviously, we still have
$$2\varphi_t \geq \psi_{t, M_t}.$$
Thus 
\begin{eqnarray*}
 \I_{\{\varphi_t<-K \}}(2\chi_t+\ddbar\varphi_t)^n  &\leq&  \I_{\{\psi_{t, M_t} <-K+\varphi_t\}}(2\chi_t+\ddbar \varphi_t )^n \\
& \leq & \I_{\{\psi_{t,M_t}<-K+\varphi_t\}}(2\chi_t+\ddbar\psi_{t,M_t})^n \\ 
&\leq &  \I_{\{\psi_{t, M_t} -A <- K \}}(2\chi_t+\ddbar\psi_{t,M_t})^n,
\end{eqnarray*}
where the second inequality follows from the comparison principle for plurisubharmonic functions. 
We now may replace $\varphi_t$ by $\psi_{t, M_t}$ in proving the proposition. By Lemma \ref{cap2} and Lemma \ref{cap1}, we have  
\begin{eqnarray*} 
\I_{\{\psi_{t,M_t}-A<-K\}}(2\chi_t+\ddbar\psi_t)^n &\leq&  K^n \ {\rm Cap}_{\chi_t}(\psi_{t,M_t}<-K+A)\\ 
&\leq&  K^n \ {\rm Cap}_{\chi_t}(\psi_t<-K+A)\\ 
&=& 
K^n \ {\rm Cap}_{\chi_t}( -(4n+2)\log (-\t_t)<-K+A+2C_1) \\
&=& K^n \ {\rm Cap}_{\chi_t} \left( \t_t <- e^{ \frac{K-A-2C_1}{4n+2}} \right) \\
&\leq& K^n \  e^{-\frac{K-A-2C_1}{4n+2}} \left( \int_{\cX_t} (-\theta_t)\chi_t ^n + n \int_{\cX_t} \chi_t^n \right).
\end{eqnarray*}
By letting $M_t \rightarrow \infty$, we have, 
by Proposition \ref{mor}, that there exists $C_2>0$ such that
$$\I_{\{\varphi_t<-K \}}(2\chi_t+\ddbar\varphi_t)^n \leq C_2K^n e^{-\frac{K-A-2C_1}{4n+2}}. $$
The proposition now follows from 
$$\I_{\{\varphi_t<-K \}}(2\chi_t+\ddbar\varphi_t)^n \geq  \sum_{m=0}^n \I_{\{\varphi_t<-K \}}(\chi_t+\ddbar\varphi_t)^m \wedge  \chi_t^{n-m} . $$

\end{proof}

\begin{corollary}  \label{expdec2} There exists   $C>0$ such that for $0 \leq m \leq n $ and all $t\in B^\circ$,
$$ \I_{\{\varphi_t<-K \}}  |\varphi_t| ( \chi_t+ \ddbar\varphi_t)^m \wedge \chi_t ^{n-m}  \leq  Ce^{-K/{(4n+4)}}.
$$

\end{corollary}

\begin{proof} We apply Proposition \ref{expdec} and there exist $C_1, C_2>0$ such that  for any $t\in B^\circ$, 
\begin{eqnarray*}
&&\I_{\{\varphi_t<-K \}} (- \varphi_t) ( \chi_t+ \ddbar\varphi_t)^m \wedge \chi_t^{n-m} \\
&=& \sum_{j=K}^\infty  \I_{-j-1\leq \varphi_t\leq -j} (-\varphi_t) ( \chi_t+ \ddbar\varphi_t)^m \wedge \chi_t ^{n-m}\\
&\leq& \sum_{j=K}^\infty (j+1) \I_{ \varphi_t\leq -j}   ( \chi_t+ \ddbar\varphi_t)^m \wedge \chi_t^{n-m}\\
&\leq& C_1 \sum_{j=K}^\infty (j+1)  j^n e^{-j/(4n+2)}\\
&\leq& C_2 \ e^{-K/(4n+4)}.
\end{eqnarray*}

\end{proof}

Finally, we remark that Proposition \ref{expdec} and Corollary \ref{expdec2} also hold uniformly for all $t\in B$. If $t\in B\setminus B^\circ$,\textcolor{black} {we can always approximate $\cX_t$ by a smooth K\"ahler manifold of dimension $n$ after normalization and resolution of singularities. The constants in the proposition and corollary are uniformly controlled.}  This can also be achieved by applying continuity of $\varphi$ on $\cX^\circ$ from section 8.

\section{Continuity of the Weil-Petersson potentials for stable families}

In this section, we will prove the continuity of the Weil-Petersson potentials for stable families of K\"ahler-Einstein manifolds. As before we consider a stable family $\pi: \cX \rightarrow B$ of $n$-dimensional K\"ahler-Einstein manifolds over a Euclidean ball $B \in \mathbb{C}^d$ whose general fibre is smooth. 
We will use the same notations as in the previous sections. 

We define the relative Weil-Petersson potential $\psi_{WP}$ by
\begin{equation}
\psi_{WP} = \frac{1}{(n+1)V} \sum_{j=0}^n \int_{\cX_t} \varphi_t (\chi_t + \ddbar \varphi_t)^j \wedge \chi_t ^{n-j}  
\end{equation} 
The the Weil-Petersson metric on $B$ is given by   (c.f. \cite{Sch, SSW})
$$\omega_{\rm WP} = \int_{\cX_t} \chi^{n+1} + \ddbar \psi_{WP}. $$

The following is the main result of this section.

\begin{proposition} \label{conti} $\psi_{\WP}$ is continuous in $B$.
\end{proposition}

It is proved in \cite{MA} that $\int_{\cX_t} \chi^{n+1}$ is a nonnegative closed $(1,1)$-current on $B$ with continuous local potentials. Then Proposition \ref{conti} immediately implies $\omega_{\rm WP}$ has continuous local potentials.

Let $\cX^{\textnormal{reg}}$ be the set of all smooth points of $\cX_t$ for $t\in B$ and let $\cX^{\textnormal{reg}}_t$ be the smooth points of $\cX_t$. Before proving Proposition \ref{conti}, we first prove that the K\"ahler-Einstein potential $\varphi$ is continuous on $\cX^{\textnormal{reg}}$, which is slightly weaker than the conclusion in Theorem \ref{main2}. The following lemma is implicitly proved in \cite{SSW} (Lemma 5.2 and Lemma 5.3).

\begin{lemma}\label{limi42} Suppose $t_j\in B$ and $t_j \rightarrow t_\infty \in B$. Then $\varphi_t|_{\cX_{t_j}^{\textnormal{reg}}}$ converges smoothly to $\varphi_{t_\infty}|_{\cX_{t_\infty}^{\textnormal{reg}}}$. 

\end{lemma}

\begin{proof} By Lemma 5.2 in \cite{SSW}, for any $k>0$ and compact set $K \subset \cX^{\textnormal{reg}}$,  $\varphi_t$ is uniformly bounded in $C^k(K\cap \cX_t)$ for all $t\in B$. Therefore for all $t_j\rightarrow t_\infty\in B$, $\varphi_{t_j}$ converges smoothly to $\varphi_{\infty}$ away from $K$, after taking a subsequence. However, by the uniform $L^\infty$-estimate, $\varphi_{\infty}$ is bounded above and bounded below by any log poles. By uniqueness of the canonical K\"ahler-Einstein current on $\cX_{t_\infty}$ (Theorem 1.1. in \cite{S}), $$\varphi_{\infty} = \varphi_{t_\infty}. $$
The lemma follows immediately.

\end{proof}

Lemma \ref{limi42} immediately implies continuity of $\varphi$ on $\cX^{\textnormal{reg}}$. 
\begin{corollary} $\varphi$ is continuous on $\cX^{\textnormal{reg}}$.

\end{corollary} 

It suffices to prove continuity of the Weil-Petersson potential at one point. We will fix $\hat{t} \in B$. By Corollary \ref{expdec2}, there exists sufficiently small $\epsilon_0>0$ such that for any $0<\epsilon<\epsilon_0$ and $t\in B$, we have 
\begin{equation}\label{asuhat}
\sum_{m=0}^n  \int_{\{ \varphi_t <- \epsilon^{-1}\}}   |\varphi_t| (\chi_t+\ddbar \varphi_t)^m\wedge (\chi_t)^{n-m} < \epsilon^3. 
\end{equation}

\begin{lemma}  \label{62} For any $0<\epsilon<\epsilon_0$,  there exists  an open neighborhood $\hat{U} \subset \cX_{\hat{t}}$ of $\cX_{\hat{t}} \setminus \cX_{\hat{t}}^{\textnormal{reg}}$, such that 
 $$\sum_{m=0}^n \int_{\hat{U}} (\chi_{\hat{t}} + \ddbar \varphi_{\hat{t}})^m \wedge (\chi_{\hat{t}})^{n-m} < \epsilon^2$$
and
$$ \{ \varphi_{\hat{t}} < -\epsilon^{-1} \} \subset  \hat{U}.$$

\end{lemma}

\begin{proof} $\varphi_{\hat{t}}$ is uniformly bounded on $\cX_{\hat{t}} \setminus \{ \varphi_{\hat{t}}< -\epsilon^{-1}\}$. Therefore there exists an open set $U$ containing $\cX_{\hat{t}}^{\textnormal{sing}} \cap \{ \varphi_{\hat{t}} \geq -\epsilon^{-1}\}$ such that 
$$\sum_{m=0}^n  \int_U    (\chi_{\hat{t}}+\ddbar \varphi_{\hat{t}})^m\wedge (\chi_{\hat{t}})^{n-m} < \epsilon^4, $$
where $\cX_{\hat{t}}^{\textnormal{sing}}= \cX_t \setminus \cX_{\hat{t}}^{\textnormal{reg}}$ is the singular set of $\cX_t$, 
since $\cX_{\hat{t}}^{\textnormal{sing}}$ is a closed pluriclosed set. 
The lemma is then proved by choosing $\hat{U} = U \cup \{ \varphi_{\hat{t}} < - \epsilon^{-1} \}$.  

\end{proof}

We choose open neighborhoods $U$ and $V$ of $\cX \setminus \cX^{\textnormal{reg}}$ in $\cX$ such that
\begin{equation}\label{conU}
\{ \varphi < - \epsilon^{-1}\}\subset \subset U\subset\subset V, ~ V\cap \cX_{\hat{t}} \subset \hat{U} .
\end{equation}
Since $\cX^{\textnormal{reg}}$ is locally a smooth holomorphic product, by partition of unity there exist an open neighborhood $B_{\delta_0}$ of $\hat{t}$ in $B$ and  a collection of finitely many smooth functions 
\begin{equation}\label{parti}
\{\rho_\alpha\}_{\alpha}
\end{equation}
 on $\cX$ satisfying
\begin{enumerate}
\item the support of $\rho_\alpha$ is biholomorphic to $\{ |t-\hat{t}|<\delta_0\}\times \mathbb{B} \subset  (\cX\setminus U)$, where $\mathbb{B}$ is a unit ball in $\mathbb{C}^n$ with $\{t\} \times \mathbb{B} \subset \cX_t$. 

\medskip

\item $0\leq \rho_\alpha \leq 1$.

\medskip

\item For any $p\in \pi^{-1}(B_{\delta_0}) \setminus V$, 
$$\sum_{\alpha} \rho_\alpha(p) = 1. $$
\end{enumerate}
Since $\varphi_t$ converges smoothly to $\varphi_{\hat{t}}$ as $t\rightarrow \hat{t}$ on $\cX^{\textnormal{reg}}$, by straightforward calculation on local Euclidean spaces, for each $\alpha$, we have 
$$\lim_{t\rightarrow \hat{t}} \int_{\cX_t} \rho_\alpha (\chi_t + \ddbar \varphi_t)^m \wedge (\chi_t)^{n-m} = \int_{\cX_{\hat{t}}} \rho_\alpha (\chi_{\hat{t}} + \ddbar \varphi_{\hat{t}})^m \wedge (\chi_{\hat{t}})^{n-m} $$
uniformly for $t\in B_{\delta_0}$. 
Since $ \cup_\alpha \{ \textnormal{supp}\ \rho_\alpha\} \subset (\cX \setminus U)$, we have 
\begin{eqnarray*}
&&\lim_{t\rightarrow \hat{t}} \int_{\cX_t\setminus U }  (\chi_t + \ddbar \varphi_t)^m \wedge (\chi_t)^{n-m} \\
&\geq &  \int_{\cX_{\hat{t}}\setminus V}   (\chi_{\hat{t}} + \ddbar \varphi_{\hat{t}})^m \wedge (\chi_{\hat{t}})^{n-m}\\
&\geq&  \int_{\cX_{\hat{t}}}   (\chi_{\hat{t}} + \ddbar \varphi_{\hat{t}})^m \wedge (\chi_{\hat{t}})^{n-m} -  \int_{ \hat{U} }  (\chi_{\hat{t}} + \ddbar \varphi_{\hat{t}})^m \wedge (\chi_{\hat{t}})^{n-m}\\
&\geq & [\chi_{\hat{t}}]^n - \epsilon^2 
\end{eqnarray*}
by Lemma \ref{62}.
In particular,  $[\chi_t]^n$ is a topological number independent of $t$ and by Proposition \ref{expdec} or the fact that $\varphi_{\hat{t}}\in \cE^1$ (c.f. \cite{BG}), we have 
$$\int_{\cX_t}  (\chi_t + \ddbar \varphi_t)^m \wedge (\chi_t)^{n-m} = [\chi_t]^n.$$
We then have the following lemma.

\begin{lemma} \label{63} For any $0<\epsilon<\epsilon_0$, there exists $0<\delta<\delta_0$, such that for any $t$  with $|t - \hat{t}|< \delta$ and $0\leq m \leq n$, we have
$$\int_{\cX_t \cap U} (\chi_t + \ddbar \varphi_t)^m \wedge (\chi_t)^{n-m} <   2\epsilon^2$$
and
$$\int_{\cX_t \setminus U} (\chi_t + \ddbar \varphi_t)^m \wedge (\chi_t)^{n-m} \geq  [\chi_{\hat{t}}]^n - 2\epsilon^2, $$
where $U$ is constructed as in (\ref{conU}). 

\end{lemma}

\begin{corollary} \label{62cor} For any $0<\epsilon<\epsilon_0$, there exists $0<\delta<\delta_0$ such that for any $t$ with $|t-\hat{t}|<\delta$, we have 
$$ \sum_{m=0}^n \int_{\cX_t \cap U} |\varphi_t|(\chi_t + \ddbar \varphi_t)^m \wedge (\chi_t)^{n-m} <  2(n+1) \epsilon.$$

\end{corollary}

\begin{proof} By Lemma \ref{63}, we have
\begin{eqnarray*}
&&  \int_{\cX_t \cap U\cap \{ \varphi_t\geq - \epsilon^{-1} \} } |\varphi_t|(\chi_t + \ddbar \varphi_t)^m \wedge (\chi_t)^{n-m} \\
&<&  \int_{\cX_t \cap U\cap \{ \varphi_t\geq - \epsilon^{-1} \} } \epsilon (\chi_t + \ddbar \varphi_t)^m \wedge (\chi_t)^{n-m}\\
&<& 2\epsilon.
\end{eqnarray*}
The corollary follows by combining the above estimate and the assumption (\ref{asuhat}).

\end{proof}

\begin{lemma} \label{64} For any $0<\epsilon<\epsilon_0$, there exists $0<\delta< \delta_0$ such that for $|t-\hat{t}|<\delta$, 
$$ \sum_{m=0}^n \left| \int_{\cX_t}  \varphi_t (\chi_t + \ddbar \varphi_t)^m \wedge (\chi_t)^{n-m}  -  \int_{\cX_{\hat{t}}}  \varphi_{\hat{t}} (\chi_{\hat{t}} + \ddbar \varphi_{\hat{t}})^m \wedge (\chi_{\hat{t}})^{n-m}  \right| <   \epsilon, $$
i.e., 
$$ |\psi_{WP}(t) - \psi_{WP}(\hat{t})|< \epsilon. $$

\end{lemma}

\begin{proof} By the same partition of unity as in (\ref{parti}), we have 
$$\lim_{t\rightarrow \hat{t}} \int_{\cX_t} \rho_\alpha \varphi_t (\chi_t + \ddbar \varphi_t)^m \wedge (\chi_t)^{n-m} = \int_{\cX_{\hat{t}}} \rho_\alpha \varphi_{\hat{t}} (\chi_{\hat{t}} + \ddbar \varphi_{\hat{t}})^m \wedge (\chi_{\hat{t}})^{n-m} $$
uniformly for $t\in B_{\delta_0}$. 
Let $U'= \cX\setminus \cup_\alpha \{ \textnormal{supp}\ \rho_\alpha\}$. Then
\begin{eqnarray*}
&&\lim_{t\rightarrow \hat{t}} \int_{\cX_t\setminus U' } \varphi_t (\chi_t + \ddbar \varphi_t)^m \wedge (\chi_t)^{n-m} \\
&\geq&  \int_{\cX_{\hat{t}}\setminus U'} \varphi_{\hat{t}} (\chi_{\hat{t}} + \ddbar \varphi_{\hat{t}})^m \wedge (\chi_{\hat{t}})^{n-m} - \epsilon^{-1} \lim_{t\rightarrow \hat{t}} \int_{\cX_t\cap  U'\cap \{ \varphi\geq - \epsilon^{-1} \}}(\chi_t + \ddbar \varphi_t)^m \wedge (\chi_t)^{n-m}\\
&& + \lim_{t\rightarrow \hat{t}} \int_{\cX_t \cap \{ \varphi < - \epsilon^{-1} \}} \varphi_t (\chi_t + \ddbar \varphi_t)^m \wedge (\chi_t)^{n-m}\\
&\geq& \int_{\cX_{\hat{t}}\setminus U'}   \varphi_{\hat{t}} (\chi_{\hat{t}} + \ddbar \varphi_{\hat{t}})^m \wedge (\chi_{\hat{t}})^{n-m}  - 3 \epsilon 
\end{eqnarray*}
and by similar argument,
\begin{eqnarray*}
&&\lim_{t\rightarrow \hat{t}} \int_{\cX_t\setminus U' } \varphi_t (\chi_t + \ddbar \varphi_t)^m \wedge (\chi_t)^{n-m} \\
&\leq&  \int_{\cX_{\hat{t}}\setminus U'} \varphi_{\hat{t}} (\chi_{\hat{t}} + \ddbar \varphi_{\hat{t}})^m \wedge (\chi_{\hat{t}})^{n-m} + \epsilon^{-1} \lim_{t\rightarrow \hat{t}} \int_{\cX_t\cap  U' \cap \{ \varphi\geq - \epsilon^{-1} \}}(\chi_t + \ddbar \varphi_t)^m \wedge (\chi_t)^{n-m}\\
&& - \lim_{t\rightarrow \hat{t}} \int_{\cX_t \cap \{ \varphi < - \epsilon^{-1} \}} \varphi_t (\chi_t + \ddbar \varphi_t)^m \wedge (\chi_t)^{n-m}\\
&\leq& \int_{\cX_{\hat{t}}\setminus U'}   \varphi_{\hat{t}} (\chi_{\hat{t}} + \ddbar \varphi_{\hat{t}})^m \wedge (\chi_{\hat{t}})^{n-m}  + 3\epsilon .
\end{eqnarray*}
On the other hand, by Corollary \ref{62cor}, for any $0<\epsilon<\epsilon_0$, there exists $\delta>0$ such that for $|t-\hat{t}| < \delta$, 
$$\int_{\cX_t \cap U}|\varphi_t | (\chi_t + \ddbar \varphi_t)^m \wedge (\chi_t)^{n-m} \leq (n+1) \epsilon.$$
Combining the above estimates,  
\begin{eqnarray*}
&& \left| \lim_{t\rightarrow \hat{t}} \int_{\cX_t  } \varphi_t (\chi_t + \ddbar \varphi_t)^m \wedge (\chi_t)^{n-m} - \int_{\cX_{\hat{t}}}  \varphi_{\hat{t}} (\chi_{\hat{t}} + \ddbar \varphi_{\hat{t}})^m \wedge (\chi_{\hat{t}})^{n-m}   \right|  \\
&\leq& 2 \epsilon + \int_{\cX_{\hat{t}}\cap U }  |\varphi_{\hat{t}}| (\chi_{\hat{t}} + \ddbar \varphi_{\hat{t}})^m \wedge (\chi_{\hat{t}})^{n-m} \\
&<&3\epsilon. 
\end{eqnarray*} 
The lemma is proved by choosing $t$ sufficiently close to $\hat{t}$.

\end{proof}

Now Proposition \ref{conti} immediately follows from Lemma \ref{64} and the result in \cite{MA} that $\int_{\cX_t} \chi^{n+1}=\ddbar \psi_\chi$ for some $\psi_\chi\in \textnormal{PSH}(B)\cap C^0(B)$.

\begin{corollary} \label{conti2} There exists $\Phi_{WP}\in \textnormal{PSH}(B)\cap C^0(B)$ such that
$$\omega_{\rm WP}= \ddbar \Phi_{WP}.$$

\end{corollary}

\section{Proof of Theorem \ref{main}}

In order to prove Theorem \ref{main}, we have to  replace the smooth base by the compactified KSB moduli space in Corollary \ref{conti2}. As in section 7 of \cite{SSW},  we will replace $\overline{\cM}_\KSB$ by a smooth model as follows. There exists a finite map
$$\phi: B \rightarrow \overline{\cM}_\KSB$$ together with a stable family
$$\pi: \cX \rightarrow B $$
such that for any $p\in \overline{\cM}_\KSB$, the fibres $(\phi\circ\pi)^{-1}(p)$ are K\"ahler-Einstein varieties correspond to the point $p$ in the moduli space. 
After normalization and resolution of singularities we can assume $B$ is smooth and $\phi$ is generically finite. We can define $\omega_{\rm WP}$ globally on $\overline{\cM}_\KSB$ by pushing forward the Weil-Petersson current $\omega_{WP,B}$ on $B$ by $\phi$. It is shown in \cite{SSW} that the Weil-Petersson current on $\overline{\cM}_\KSB$ has bounded local potentials and is independent of the choice of $B$. Furthermore, $\omega_{\rm WP}$ is the curvature of the CM line bundle $\cL_{CM} \rightarrow \overline{\cM}_\KSB$.
Let $h$ be the smooth hermitian metric on $\cL_{CM}$ with $\omega_{\rm FS}=\textnormal{Ric}(h)$ being the Fubini-Study metric induced by a projective embedding from the linear system of a sufficiently large power of  $\cL_{CM}$.
Then by Corollary \ref{conti2}, we have 
$$\phi^*\omega_{WP, B} = \phi^*\omega_{\rm FS} + \ddbar \varphi_{WP, B}$$
for some $\varphi_{WP, B}\in \textnormal{PSH}(B, \phi^*\omega_{\rm FS})\cap C^0(B)$. In particular, it is proved in \cite{SSW} that $\varphi_{WP, B}$ coincides with the pullback of a bounded function $\varphi_{WP} \in \textnormal{PSH}(\overline{\cM}_\KSB, \omega_{\rm FS})\cap L^\infty(\overline{\cM}_\KSB)$  on the regular set of $\pi$. 
Let $F$ be an irreducible component of a fibre of $\pi:B \rightarrow \overline{\cM}_\KSB$ with $\dim F\geq 1$. Obviously, $\phi^*\cL_{CM}|_F =0$ and $\omega_{\rm FS}|_F=0$. This implies that $\varphi_{WP, B}$ is plurisubharmonic and continuous on $F$, hence $\varphi_{WP, B}$ must be constant along $F$. Therefore $\varphi_{WP, B}$ can descend to a continuous function on $\overline{\cM}_\KSB$. This completes the proof of Theorem \ref{main}.

\section{Proof of Theorem \ref{main2}}

In this section, we will complete the proof of Theorem \ref{main2}.
Let $\cX\ra B$ be a stable family of canonically polarized manifolds,  where $B\sub\C^d$ is a ball.
Let $\cE_{slc}$ be the set of all non-klt locus of $\cX_t$ for $t\in B$. 

 For $t\in B^\circ$,  let $\varphi_t$ satisfy $\o_t^n=(\chi+\ddbar\varphi_t)^n=e^{\varphi_t}\O_t$.  The following is the main result of of this section to bound the volume of the level sets of the K\"ahler-Einstein potentials $\varphi_t$.  We let $g_t$ be the K\"ahler-Einstein metric associated to $\o_t$.
 
 The following theorem is proved in \cite{SSW}.
 
\begin{theorem} For any $t_j \in B^\circ$ with $t_j \rightarrow t_\infty$, $(\cX_{t_j}, g_{t_j})$ converges in pointed Gromov-Hausdorff topology to the metric completion of $(\cX_{t_\infty}^\circ, g_{t_\infty})$, where $\cX_{t_\infty}^\circ$ is the smooth part of $\cX_{t_\infty}$ and $g_{t_\infty}$ is the smooth K\"ahler-Einstein metric on  $\cX_{t_\infty}^\circ$ that extends to the unique K\"ahler-Einstein current on $\cX_{t_\infty}$. In particular, the metric completion of $(\cX_{t_\infty}^\circ, g_{t_\infty})$ is homeomorphic to $\cX_{t_\infty}\setminus \cE_{scl}$ and the convergence is smooth on $(\cX_{t_\infty}^\circ, g_{t_\infty})$.

\end{theorem}

In fact, it is proved in \cite{S} (that $\varphi_t$ is always bounded away from $\cE_{slc}$. 

\begin{lemma} \label{0bound81} Let $p_t\in \cX_t$  be a continuous section of $\pi: \cX^\circ \rightarrow B^\circ$ such that 
$$\{p_t\}_{t\in B} \cap \cE_{slc} = \phi.$$ 
Then for any $R>0$, there exists $C_R>0$ such that for any $t\in B$, 
$$\sup_{B_{g_t}(p_t, R)} |\varphi_t|  \leq C_R. $$

\end{lemma}

\begin{proof} We prove by contradiction. By assumption, geodesic unit balls centered at $p_t$ are uniformly non-collapsed, i.e., there exists $c>0$ such that for all $t\in B^\circ$, 
$$Vol_{g_t}(B_{g_t}(p_t, 1)) \geq c. $$

Suppose there exist a sequence $t_j \in B^\circ$ with $t_j \rightarrow t_\infty \in B$ and a sequence $x_j \in \cX_{t_j}$ with $x_j \rightarrow x_\infty \in \cX_{t_\infty} \setminus \cE_{slc}$ in Gromov-Hausdorff distance such that for any $A>0$, there exists $J>0$ so that for all $j \geq J$, we have
$$\varphi_{t_j} (x_j)< - A.$$
We can assume that for some fixed sufficiently small $r>0$, $B_{g_{t_j}}(x_j, r)$ converges in Gromov-Hausdorff topology to $B_{g_{t_\infty}}(x_\infty, r)$. By the local $L^2$-estimate and partial $C^0$-estimate in \cite{SSW}, there exists a sequence of $\sigma_j \in H^0(\cX_{t_j}, mK_{\cX_{t_j}})$ for some fixed $m>>1$ satisfying the following.
\begin{enumerate}

\item For each $j$, $$\int_{\cX_{t_j}}|\sigma_j|^2_{h_{t_j}^m} dV_{g_{t_j}} = 1, $$
where $h_t = (\omega_t^n)^{-1}$ is the hermitian metric on the relative canonical bundle $K_{\cX/ B}$ induced by the K\"ahler-Einstein metrics on the fibres. 

\item There exists $C>0$ such that for each $j$, $$\sup_{B_{g_{t_j}}(x_j, r)} |\nabla \sigma_j|^2_{h_{t_j}^m, g_{t_j}} \leq C.$$

\item There exists $c>0$ such that for each $j$, $$|\sigma_j|^2_{h_{t_j}^m} (x_j) \geq c.$$

\end{enumerate}

By passing to a subsequence, we can assume that $\sigma_j$ converges to an $L^2$-integrable pluricanonical section $\sigma_\infty$ on $B_{g_{t_\infty}}(x_\infty, r)$.  Since $h_{t_j}$ converges smoothly to $h_{t_\infty}$ on the smooth part of $\cX_{t_\infty}$,   by the uniform gradient estimate, $|\sigma_\infty|^2_{h_{t_\infty}^m}$ extends to a continuous function on $B_{g_{t_\infty}}(x_\infty, r)$ and there exist $\epsilon>0$ and  $0<2\delta<r$ such that 
$$\inf_{B_{g_{t_\infty}}(x_\infty, 2\delta)} |\sigma_\infty|^2_{h_{t_\infty}^m} > \epsilon. $$
We fix a projective embedding $\Phi: \cX \rightarrow \mathbb{CP}^{N_m}$ by $mK_{\cX/ U}$ for an affine neighborhood $U$ of $t_\infty$ in $B$ such that  $\chi= m^{-1} \ddbar \log ( 1+\sum_{i=1}^{N_m} |z_i|^2  )$ in an affine and bounded open set of $\mathbb{CP}^{N_m}$ containing $B_{g_{t_\infty}}(x_\infty, r)$ and $B_{g_{t_j}}(x_j, r)$ for all $j$.  Then there exist $a_{j, i}\in \mathbb{C}$ and $C>0$ such that $$\sigma_j = a_{j, 0}+ \sum_{i=1}^{N_m} a_{j, i} z_i$$ and for all $i$, $j$, we have 
$$|a_{i, j}| \leq C$$
and
$$\lim_{j\rightarrow \infty} a_{j, i} = a_{\infty, i}.$$
In particular, $a_{\infty, 0} \neq 0$. 
Since $$|\sigma_j|^2_{h_{t_j}^m} = \left( \frac{a_{j, 0} + \sum_{i=1}^{N_m} a_{j, i} |z_i|^2}{ 1+\sum_{i=1}^{N_m} |z_i|^2} \right) e^{-m\varphi_{t_j}} $$
is uniformly bounded above for in $B_{g_{t_j}}(x_j, r)$, $\varphi_{t_j}$ must be uniformly bounded from below in $B_{g_{t_j}}(x_j, r)$. This leads to contradiction. 

\end{proof}

The proof of Lemma \ref{0bound81} already implies the continuity of $\varphi$. However, we include the following two estimates as first and second order control of $\varphi_t$.

\begin{lemma} \label{2ndest} For any relative compact  $K \subset\subset \cX \setminus \cE_{scl}$, there exists $C_K >0$ such that
\begin{equation}
\sup_K tr_{g_t}(\chi) \leq C_K. 
\end{equation}

\end{lemma}

\begin{proof} We will pick base points $p_t$ for $t\in B^\circ$ and we can assume that there exists $c>0$ such that %
$$Vol(B_{g_t}(p_t, 1)) \geq c$$
for all $t\in B^\circ$. Let $r_t(x)$ be the distance function from $x$ to $p_t$ on $\cX_t$. Immediately we have 
$$ |\nabla r_t |_{g_t}=1, ~~ \Delta_t  r_t \leq ( r_t^{-1} + 1). $$
We define the cut-off function $\phi_{t, R} (x) = \rho (r_t(x))$ satisfying 
$$ \rho_R (r) = 1, ~ if ~ r \leq R, ~ \rho_R(r) =0, ~ if ~ r\geq 2R, $$
where smooth nonnegative decreasing function $\rho$ satisfies
$$\rho_R \geq 0, ~ 0 \leq \rho_R^{-1} (\rho_R')^2 \leq C R^{-2}, ~ |\rho_R''| \leq C R^{-2} $$
for a fixed constant  $C>0$. 

We let $$ H_{t, R} = \phi_R \left( \log tr_{g_t}(\chi_t) - 3A \varphi_t \right). $$
Straightforward calculations show that for some fixed and sufficiently large $A>>1$, we have 
\begin{eqnarray*}
&&\Delta_t H_{t, R}\\
&\geq& \phi_R \left( -3An - \frac{|\nabla tr_{g_t}(\chi_t)|^2}{tr_{g_t}(\chi_t)} + 2A tr_{g_t}(\chi_t) \right) - 2Re\left( \nabla \phi_R \cdot  \frac{\nabla tr_{g_t}(\chi_t)}{tr_{g_t}(\chi_t)} \right) \\
&&+ \left( \Delta_t \phi_R \right) \left( \log tr_{g_t}(\chi_t) - 3A \varphi_t \right)\\
&\geq & 2A \phi_R tr_{g_t}(\chi_t) - A\left( \phi_R tr_{g_t}(\chi_t) \right)^{-1} - 4An - A\log tr_{g_t}(\chi_t). 
\end{eqnarray*}
Since  $xlogx \leq e$ for $x>0$ and $\varphi_t$ is uniformly bounded in the support of $\phi_R$ for all $t$, we can apply the maximum principle and conclude that $H_{t, R}$ is uniformly bounded for all $t\in B^\circ$ for any fixed $R\geq 1$. The lemma then immediately follows.

\end{proof}

We remark that Lemma \ref{2ndest} can also be used to prove continuity of the Weil-Petersson potentials in section 6.

\begin{proposition} \label{gradest} For any relatively compact  $K \subset\subset \cX \setminus \cE_{scl}$, there exists $C_K >0$ such that
\begin{equation}
\sup_K |\nabla \varphi_t|_{g_t} \leq C_K. 
\end{equation}

\end{proposition}

\begin{proof} The proposition can be proved by the partial $C^0$-estimate and $L^2$-estimate. We give an alternative proof using the maximum principle.  Let $\Delta_t$ be the Laplace operator of $g_t$ on $\cX_t$, where $t\in B^\circ$. Straightforward calculations show that
\begin{eqnarray*}
&&\Delta_t \left( |\nabla\varphi_t|^2_{g_t} \right) \\
&=&- |\nabla\varphi_t|^2_{g_t}+ |\nabla^2 \varphi_t|^2_{g_t} + |\nabla\overline{\nabla}\varphi_t|^2_{g_t} + 2Re \left( \nabla(tr_{g_t}(\chi_t) \cdot \nabla \varphi_t \right)_{g_t}.
\end{eqnarray*}

We will pick base points $p_t$ for $t\in B^\circ$  and the cute-off function $\phi_R$ as in the proof of Lemma \ref{2ndest}.   By the $C^0$-estimate of $\varphi_t$, there exists $A_R>1$ such that
$$2||\varphi_t||_{L^\infty(B_{g_t}(p_t, 2R)} \leq A_R$$
for all $t\in B^\circ$. In $B_{g_t}(p_t, 2R)$, we have
\begin{eqnarray*}
&&\Delta_t \left( \frac{|\nabla\varphi_t|^2_{g_t}}{A - \varphi_t} \right) \\
&=&\frac{ \Delta_t \left( |\nabla  \varphi_t|^2\right)}{A_R - \varphi_t} + \frac{ |\nabla\varphi_t|^2_{g_t} \left( \Delta_t \varphi_t \right) }{(A_R-\varphi_t)^2} + 2Re \left( \nabla |\nabla \varphi_t|^2 \cdot \frac{\nabla \varphi_t}{(A_R- \varphi_t)^2} \right)_{g_t}+ 2\frac{|\nabla\varphi_t|^4_{g_t}}{(A - \varphi_t)^3}\\
&=&\frac{- |\nabla\varphi_t|^2_{g_t}+ |\nabla^2 \varphi_t|^2_{g_t} + |\nabla\overline{\nabla}\varphi_t|^2_{g_t} + 2Re \left( \nabla(tr_{g_t}(\chi_t) \cdot \nabla \varphi_t \right)_{g_t}}{A_R - \varphi_t} \\
&&+ \frac{ |\nabla\varphi_t|^2_{g_t} \left( n - tr_{g_t}(\chi_t) \right) }{(A_R-\varphi_t)^2} + 2Re \left( \nabla |\nabla \varphi_t|^2 \cdot \frac{\nabla \varphi_t}{(A_R- \varphi_t)^2} \right)_{g_t} + 2\frac{|\nabla\varphi_t|^4_{g_t}}{(A - \varphi_t)^3}.
\end{eqnarray*}
Also we have the following calculations because the Fubini-Study metric $\chi_t$ in the ambient projective space has bounded curvature 
$$\Delta_t tr_{g_t}(\chi_t) \geq |\nabla tr_{g_t}(\chi_t)|^2_{g_t} - tr_{g_t}(\chi_t) - C \left(tr_{g_t}(\chi_t)\right)^2.$$
By the result in \cite{SSW}, $tr_{g_t}(\chi_t)$ is uniformly bounded away from $\cE_{slc}$ and so there exists $A'_R>0$ such that in $A'_{g_t}(p_t, 2R)$, we have
$$\Delta_t tr_{g_t}(\chi_t) \geq   |\nabla tr_{g_t}(\chi_t)|^2_{g_t} - A'_R.$$

Now we let $$H_{t, R} = \phi_R \left( \frac{|\nabla\varphi_t |^2_{g_t}} {A_R - \varphi_t} + B\ tr_{g_t} (\chi_t) \right), ~K_{t, R} =   \left( \frac{|\nabla\varphi_t |^2_{g_t}} {A_R - \varphi_t} +  B\ tr_{g_t} (\chi_t) \right)$$
for some sufficiently large $C>0$ to be determined. 
Then there exist $C_1, C_2, ..., C_5 >0$ such that on $\cX_t$ for all $t\in B^\circ$, we have
\begin{eqnarray*}
&&\Delta_t H_{t, R} \\
&\geq&  \left( \Delta_t \phi_R \right) K_{t, R} + 2Re\left( \nabla\phi_R \cdot \nabla \left(\frac{|\nabla\varphi_t|^2_{g_t}}{A - \varphi_t}\right) \right)_{g_t} + \phi_R  \Delta_t \left( \frac{|\nabla\varphi_t|^2_{g_t}}{A - \varphi_t}     + C   tr_{g_t}(\chi_t) \right) \\
&\geq& \left( \Delta_t \phi_R \right) K_{t, R} + 2 Re \left( \nabla \phi_R, \nabla K_{t, R} - B\ \nabla tr_{g_t}(\chi_t) \right)_{g_t} +\phi_R \frac{ |\nabla\varphi_t|^2_{g_t} \left( n - tr_{g_t}(\chi_t) \right) }{(A_R-\varphi_t)^2} \\
&& +\phi_R \left( \frac{  - |\nabla\varphi_t|^2_{g_t}+ |\nabla^2 \varphi_t|^2_{g_t} + |\nabla\overline{\nabla}\varphi_t|^2_{g_t} + 2Re \left( \nabla(tr_{g_t}(\chi_t) \cdot \nabla \varphi_t \right)_{g_t} }{A_R - \varphi_t}\right)\\
&&+(2-2\epsilon)\phi_R Re\left( \nabla K_{t, R} \cdot \frac{\nabla \varphi_t}{A-\varphi_t}   \right)_{g_t} + 2\epsilon \phi_R \frac{|\nabla\varphi_t|^4_{g_t}}{(A - \varphi_t)^3} + B\ \phi_R \Delta_t tr_{g_t}(\chi_t) \\
&\geq& (2-2\epsilon) Re\left( \nabla H_{t, R} \cdot \frac{\nabla \varphi_t}{A-\varphi_t}   \right)_{g_t}+ 2\epsilon \phi_R \frac{|\nabla\varphi_t|^4_{g_t}}{(A - \varphi_t)^3} - C_1\frac{|\nabla\varphi_t|^2_{g_t}}{A - \varphi_t}- C_2 \frac{|\nabla\varphi_t|^2_{g_t}}{(A - \varphi_t)^2} - C_3\\
&\geq& (2-2\epsilon) Re\left( \nabla H_{t, R} \cdot \frac{\nabla \varphi_t}{A-\varphi_t}   \right)_{g_t}+ 2\epsilon   \frac{H_{t, R}^2}{A - \varphi_t} - C_4\frac{H_{t, R}}{A - \varphi_t} - C_5.
\end{eqnarray*}
By applying the maximum principle, there exists $C_R>0$ such that for all $t\in B^\circ$, 
$$\sup_{\cX_t} H_{t, R} \leq C_R. $$
This completes the proof of the proposition.

\end{proof}

The following corollary is an immediate consequence of the uniform gradient estimate in Proposition \ref{gradest} as $\varphi_t$ is uniformly Lipschitz in any compact set $\cX_t \cap \cX^\circ$.  
\begin{corollary} $\varphi_t \in C^0(\cX \setminus \cE_{scl})$.

\end{corollary}

\end{document}